\documentclass[12pt,reqno]{amsart}

\usepackage{amsmath,amsfonts,amssymb,amsthm,mathtools}
\usepackage{hyperref}
\usepackage{gensymb}
\usepackage{graphicx}	
\usepackage{float}
\usepackage{color}
\usepackage{longtable}
\usepackage{pdfpages}
\usepackage{footnote}
\usepackage{appendix}
\usepackage{sidecap}
\usepackage{tabularx}
\usepackage{framed}
\makesavenoteenv{tabular}
\makesavenoteenv{table}
\usepackage{array}
\usepackage{booktabs}

\usepackage{tikz}
\usepackage{tikz-cd}
\usetikzlibrary{shapes.geometric,calc,positioning,3d,intersections,arrows}
\usetikzlibrary{decorations.markings}
\usepackage{pgfplots}

\calclayout

\renewenvironment{framed}[1][\hsize]
   {\MakeFramed{\hsize#1\advance\hsize-\width \FrameRestore}}%
   {\endMakeFramed}

\newcommand\A{{\mathcal A}}
\newcommand{\la}{\langle}
\newcommand{\ra}{\rangle}

\newcommand\ba{{\boldsymbol a}}
\newcommand\bb{{\boldsymbol b}}
\newcommand\bc{{\boldsymbol c}}
\newcommand\bd{{\boldsymbol d}}
\newcommand\bp{{\boldsymbol p}}
\newcommand\bq{{\boldsymbol q}}
\newcommand\bx{{\boldsymbol x}}
\newcommand\by{{\boldsymbol y}}
\newcommand\bz{{\boldsymbol z}}
\newcommand\bu{{\boldsymbol u}}

\newcommand\bT{{\boldsymbol T}}
\newcommand\C{{\mathbb C}}
\newcommand\g{\mathfrak{g}}

\newcommand\op[1]{\mathop{\rm #1}\nolimits}
\newcommand\p{\partial}

\newcommand\R{{\mathbb R}}
\newcommand\cF{{\mathcal F}}
\newcommand\cR{{\mathcal R}}
\newcommand\cS{{\mathcal S}}

\newcommand\Z{{\mathbb Z}}
\newcommand\com[1]{}

\newtheorem{thm}{Theorem}[section]
\theoremstyle{definition}

\newtheorem{theorem}[thm]{Theorem}

\newtheorem{proposition}[thm]{Proposition}

\newtheorem*{remark}{Remark}

\pgfplotsset{compat=1.15}

\begin{document}
\title{Joint Invariants of Linear Symplectic Actions}
\author{Fredrik Andreassen}
\author{Boris Kruglikov}
 \address{Institute of Mathematics and Statistics, UiT the Arctic University of Norway, Troms\o\ 90-37, Norway.
E-mails: \textsc{fredrik.andreassen@gmail.com,\ boris.kruglikov@uit.no}.}
 \maketitle

\begin{abstract}
We review computations of joint invariants on a linear symplectic space,
discuss variations for an extension of group and space
and relate this to other equivalence problems and approaches, most
importantly to differential invariants.
\end{abstract}

 \section{Introduction}\label{S1}

The classical invariant theory \cite{H,PV,Ol} investigates
polynomial invariants of linear actions
of a Lie group $G$ on a vector space $V$, i.e.\
describes the algebra $(S\,V^*)^G$. For instance, the case of binary forms
corresponds to $G=\op{SL}(2,\C)$ and $V=\C^2$; equivalently for
$G=\op{GL}(2,\C)$ one studies instead the algebra of relative invariants.
The covariants correspond to invariants in the tensor product $V\otimes W$
for another representation $W$. Changing to the Cartesian product
$V\times W$ leads to joint invariants of $G$.

In this paper we discuss joint invariants corresponding to the
(diagonal) action of $G$ on the iterated Cartesian product $V^{\times m}$
for increasing number of copies $m\in\mathbb{N}$.
We will focus on the case $G=\op{Sp}(2n,\R)$, $V=\R^{2n}$ and
discuss the conformal $G=\op{CSp}(2n,\R)=\op{Sp}(2n,\R)\times\R_+$ and
affine $G=\op{ASp}(2n,\R)=\op{Sp}(2n,\R)\ltimes\R^{2n}$ versions later.

This corresponds to invariants of $m$-tuples of points in $V$, i.e.\
finite ordered subsets.
By the Hilbert-Mumford \cite{H} and Rosenlicht \cite{R} theorems, the
algebra of polynomial invariants (for the semi-simple $G$) or the field
of rational invariants (in all other cases considered) can be interpreted
as the space of functions on the quotient space $V^{\times m}/G$.

For $G=\op{Sp}(2n,\C)$ the algebra of invariants is known \cite{W}.
Generators and relations (syzygies) are described in
the first and the second fundamental theorems respectively.
We review this in Theorem \ref{T1} (real version),
and complement by explicit examples of free resolutions of the algebra.
In addition, we describe the field of rational invariants.

We also discuss invariants with respect to the group
$G=\op{Sp}(2n,\R)\times S_m$, in which case considerably less is known.
Another generalization we consider is the field of invariants for the
conformal symplectic Lie group $G=\op{CSp}(2n,\R)$ on the contact space.

When approaching invariants of infinite sets, like curves or
domains with smooth boundary, the theory of joint invariants is
not directly applicable and the equivalence problem is solved via
differential invariants \cite{KL}. In the case of a group $G$ and 
a space $V$ as above this problem was solved in \cite{JK}. We claim 
that the differential invariants from this reference
can be obtained in a proper limits of joint invariants, i.e.\
via a certain discretization and quasiclassical limit,
and demonstrate it explicitly in several cases.

In this paper we focus on discussion of various interrelations
of joint invariants. In particular, at the conclusion
we note that joint invariants can be applied
to the equivalence problem of binary forms. Since these have been
studied also via differential invariants \cite{BL,Ol} a further link
to the above symplectic discretization is possible.

The relation to binary forms mentioned above is based on
the Sylvester theorem \cite{S},
which in turn can be extended to more general Waring decompositions,
important in algebraic geometry \cite{AH}.
Our computations should carry over to the general case.
This note is partially based on the results of \cite{FA},
generalized and elaborated in several respects.

\section{Recollection: Invariants}\label{S2}

We briefly recall the basics of invariant theory,
referring to \cite{MFK,PV} for more details.

Let $G$ be a Lie group acting on a manifold $V$. A point $x\in V$ is
regular if a neighborhood of the orbit $G\cdot x$ is fibred by $G$-orbits.
A point $x\in V$ is weakly regular, if its (not necessary $G$-invariant)
neighborhood is foliated by the orbits of the Lie algebra $\g=\op{Lie}(G)$.
In general the action can lack regular points, but a generic point
is weakly regular. For algebraic actions a Zariski open set of points is regular.

\subsection{Smooth invariants}

If $G$ and $V$ are only smooth (and non-compact),
there is little one can do to guarantee regularity a priori.
An alternative is to look for local invariants, i.e.\
functions $I=I(x)$ in a neighborhood $U\subset V$
such that $I(x)=I(g\cdot x)$ as long as $x\in U$ and $g\in G$ satisfy
$g\cdot x\in U$.

The standard method to search for such $I$ is by elimination of group
parameters, namely by computing quasi-transversals \cite{PV} or
using normalization and moving frame \cite{Ol}. Another way is
to solve the linear PDE system $L_\xi(I)=0$ for $\xi\in\g=\op{Lie}(G)$.

Given the space of invariants $\{I\}$ one can extend $U\subset V$ and
address regularity. In our case the invariants are easy to compute and
we do not rely on any of these methods, however instead we describe
the algebra and the field of invariants depending on specification of
the type of functions $I$.

\subsection{Polynomial invariants}

If $G$ is semi-simple and $V$ is linear, then by the Hilbert-Mumford theorem
generic orbits can be separated by polynomial invariants $I\in (S\,V^*)^G$,
where $S\,V^*=\oplus_{k=0}^\infty S^kV^*$ is the algebra of homogeneous polynomials on $V$.
With a choice of linear coordinates $\bx=(x_1,...,x_n)$ on $V$
we identify $S\,V^*=\mathbb{R}[\bx]$.

Moreover, by the Hilbert basis theorem, the algebra of polynomial invariants
$\A_G=(S\,V^*)^G$ is Noetherian, i.e.\ finitely generated by some
$\ba=(a_1,\dots,a_s)$, $a_j=a_j(\bx)\in\A_G$.

Denote by $\cR=\R[\ba]$ the free commutative $\R$-algebra generated 
by $\ba$. It forms a free module $F_0$ over itself.
$\A_G$ is also an $\cR$-module with surjective $\cR$-homomorphism
$\phi_0:F_0\to\A_G$, $\phi_0(a_j)=a_j(\bx)$.
The first syzygy module $S_1=\op{Ker}(\phi_0)$ fits the exact sequence		
 \[
0 \to S_1 \to F_0 \to \A_G \to 0.
 \]	
A {\it syzygy} is an element of $S_1$, i.e.\ a relation $r=r(\ba)$
between the generators of $\A_G$ of the form
$\sum_{p=1}^k r_{i_p}a_{j_p}=0$, $r_{i_p}\in\cR$.

The module $S_1$ is Noetherian, i.e.\ finitely generated by some
$\bb=(b_1,\dots,b_t)$. Denote the free $\cR$-module generated by
$\bb$ by $F_1=\cR[\bb]$. The natural homomorphism
$\phi_1:F_1\to S_1\subset F_0$, $\phi_1(b_j)=b_j(\ba)$, defines
the second syzygy module $S_2=\op{Ker}(\phi_1)$, and we can continue
obtaining $S_2\subset F_2=\cR[\bc]$, etc.
This yields the exact sequence of $\cR$-modules:
 \[
\dots \xrightarrow{\phi_3} F_2 \xrightarrow{\phi_2} F_1
\xrightarrow{\phi_1} F_0 \xrightarrow{\phi_0} \A_G \to 0.
 \]

The Hilbert syzygy theorem states that $q$-th module of syzygies $S_q$
is free for $q\ge s=\#\ba$. In particular, the minimal free
resolution exists and has length $\leq s$, see \cite{E}.

To emphasize the generating sets we depict free resolutions as follows:
 \[
\R[\bx] \supset \A_G \gets \R[\ba] \gets
\cR[\bb] \gets \cR[\bc] \gets \dots\gets 0.
\]

\subsection{Rational invariants}

If $G$ is algebraic, in particular reductive, then by the Rosenlicht theorem
\cite{R} generic orbits can be separated by rational invariants $I\in\cF_G$.
Here $\R(\bx)$ is the field of rational functions on $V$
and $\cF_G=\R(\bx)^G$.

Let $d$ be the transcendence degree of $\cF_G$. This means that
there exist $(a_1,\dots,a_d)=\bar{\ba}$, $a_j\in\cF_G$,
such that $\cF_G$ is an algebraic extension of $\R(\bar{\ba})$.
Then either $\cF_G=\R(\ba)$ for $\ba=\bar{\ba}$ or
$\cF_G$ is generated by a set $\ba\supset\bar{\ba}$,
which by the primitive element theorem can be assumed of cardinality
$s=\#\ba=d+1$, i.e.\ $\ba=(a_1,\dots,a_d,a_{d+1})$.
In the latter case there is one algebraic relation on $\ba$.
Note that $d\leq n$ because $\R(\bar{\ba})\subset\R(\bx)$.

We adopt the following convention for depicting this:
 \[
\R(\bx) \supset \cF_G \stackrel{\text{alg}}\supset \R(\bar{\ba}) \overset{d}{\supset} \R.	
 \]

\subsection{Our setup}

If the Lie group $G$ acts effectively on $V$, then for some $q$
it acts freely on $V^{\times q}$, and hence on all $V^{\times m}$ for
$m\ge q$.
 \com{
(The proof of this claim is inductive: the orbits of $G$ on $V$ are
non-constant, choose $x_1$ in weakly regular (= maximal dimension) orbit.
Its stabilizer $H_1=G_{x_1}$ acts effectively on $V$, but since the
first component of $\bx$ is fixed, we reduce to the action on
$V^{\times(q-1)}$. Choose $x_2$ in the second component weakly regular,
let $H_2=(H_1)_{x_1}=G_{x_1,x_2}$ and continue. The procedure will stop
as the dimension of the groups strictly decreases.)
 }
The number of rational invariants separating a generic orbit in
$V^{\times m}$ is equal to the codimension of the orbit.

It turns out that knowing all those invariants $I$
on $V^{\times q}$ is enough to generate
the invariants on $V^{\times m}$ for $m>q$. Indeed, let
$\pi_{i_1,\dots,i_q}:V^{\times m}\to V^{\times q}$ be the projection
to the factors $(i_1,\dots,i_q)$. Then the union of
$\pi^*_{i_1,\dots,i_q}I$ for $I$ from the field $\cF_G(V^{\times q})$
gives the generating set of the field $\cF_G(V^{\times m})$,
and similarly for the algebra of invariants.

Below we denote $\A_G^m=\A_G(V^{\times m})$ and $\cF_G^m=\cF_G(V^{\times m})$.

\subsection{The equivalence problem}

For a semi-simple Lie group $G$ the field $\cF_G$ is obtained from
the ring $\A_G$ by localization (field of fractions): $\cF_G=F(\A_G)$.
Hence we discuss a solution to the equivalence problem through rational invariants.

Let $I_1,\dots, I_s$ be the generating set of invariants of the action of
$G$ on $V^{\times q}$. If $s=d+1$, this set of generators is subject
to an algebraic condition, which constrains the generators to
an algebraic set $\Sigma\subset\R^s$. This is the signature space, cf.\ \cite{O2}.

Now the $q$-tuple of points $X=(\bx_1,\dots,\bx_q)$
is mapped to $I_1(X),\dots,I_s(X)\in\Sigma$. Denote this map by $\Psi$.
Two generic configurations of points $X',X''\in V^{\times q}$
are $G$-equivalent iff their signatures coincide $\Psi(X')=\Psi(X'')$.

\section{Invariants on Symplectic Vector Spaces}

Let $V=\R^{2n}(x^1,\dots,x^n,y^1,\dots,y^n)$ be equipped with the standard symplectic form $\omega =dx^1\wedge dy^1+\dots+dx^n\wedge dy^n$.
The group $G=\op{Sp}(2n,\R)$ acts almost transitively on $V$,
preserving the origin $O$.
Thus there are no continuous invariants of the action, $\cF_G^1=\R$.
The first invariant occurs already for two copies of $V$.
Namely for a pair of points $A_i,A_j\in V$
the double symplectic area of the triangle $OA_iA_j$ is
 \[
a_{ij}=\omega(OA_i,OA_j)= \bx_i\by_j-\bx_j\by_i =\sum^n_{k=1} x^k_iy^k_j-x^k_jy^k_i.
 \]

\subsection{The case $n=1$}\label{F:n=1}

Consider at first the case of dimension 2, where $V=\R^2(x,y)$,
$\omega = dx \wedge dy$. The invariant $a_{12}=x_1y_2-x_2y_1$ on $V\times V$
generates pairwise invariants $a_{ij}$ on $V^{\times m}$ for $m\ge2$
induced through the pull-back of the projection
$\pi_{i,j}:V^{\times m}\to V\times V$ to the corresponding factors.
Below we describe minimal free resolutions of $\A_G^m$ for $m\ge2$.

\subsubsection{$V\times V$}

Here the algebra is generated by one element, whence the resolution:
 \[
\R[x_1,x_2,y_1,y_2] \supset \A_G^2 \gets \R[a_{12}] \gets 0
 \]
In other words, $\A_G^2\simeq\cR:=\R[a_{12}]$. Note that
$\cF_G^2=\R(a_{12})$.

\subsubsection{$V^{\times 3}=V\times V\times V$}

Here the action is free on the level of $m=3$ copies of $V$ and we
get $3=\dim V^{\times 3}-\dim G$ independent invariants
$a_{12}$, $a_{13}$, $a_{23}$. They generate the entire algebra,
and we get the following minimal free resolution:
 \[
\R[x_1,x_2,x_3,y_1,y_2,y_3] \supset\A_G^3\gets\R[a_{12},a_{13},a_{23}]\gets0
 \]
Once again, $\A_G^3\simeq \cR:=\R[a_{12},a_{13},a_{23}]$. Also
$\cF_G^3=\R(a_{12},a_{13},a_{23})$.

\subsubsection{$V^{\times 4}$}

Here $\dim V^{\times 4}=8$, $\dim G=3$ and we have 6 invariants
$\ba=\{a_{ij}:1\leq i<j\leq4\}$. To obtain a relation,
we try eliminating the variables $x_1,x_2,x_3,x_4,y_1,y_2,y_3,y_4$,
but this fails with the standard \textsc{Maple} command.
Yet, using the transitivity of the $G$-action
we fix $A_1$ at $(1,0)$ and $A_2$ at $(0,a_{12})$, and then
obtain the only relation
 \[
b_{1234} := a_{12}a_{34} - a_{13}a_{24} + a_{14}a_{23} = 0
 \]
that we identify as the \textit{Pl\"ucker relation}.
Thus the first syzygy is a module over $\cR:=\R[\ba]$ with one generator,
hence the minimal free resolution is:
 \[
\R[\bx,\by] \supset \A_G^4 \gets \R[a_{12},a_{13},a_{14},a_{23},a_{24},a_{34}] \gets \R[b_{1234}] \gets 0.
 \]
For the field of rational invariants one of the generators is superfluous,
for instance we can resolve the relation $b_{1234} = 0$ for
$a_{34}=(a_{13}a_{24}-a_{14}a_{23})/a_{12}$, and get
 \[
\mathbb{R}(x_1,x_2,x_3,x_4,y_1,y_2,y_3,y_4) \supset
\cF_G^4 \simeq \mathbb{R}(a_{12},a_{13},a_{14},a_{23},a_{24}) \overset{5}{\supset} \mathbb{R}	
 \]

\subsubsection{$V^{\times 5}$}

The algebra of invariants $\A_G^5$ is generated by
$\ba=\{a_{ij}:1\leq i<j\leq5\}$. This time the number
of generators is 10, while codimension of the orbit is $10-3=7$.
Using the same method we obtain that the first syzygy module
is generated by the Pl\"ucker relations
 \[
b_{ijkl} := a_{ij}a_{kl} - a_{ik}a_{jl} + a_{il}a_{jk}=0.
 \]
We have 5 of those: $\bb=\{b_{ijkl}:1\leq i<j<k<l\leq5\}$.
Thus there should be relations among relations,
or equivalently second syzygies.
If $F_0=\R[\ba]=:\cR$	and $F_1=\cR[\bb]$ then this module
is $S_2=\op{Ker}(\phi_1:F_1\to S_1\subset F_0)$.
Using elimination of parameters we find that $S_2$ is generated by			
$\bc=\{c_i:1\leq i\leq5\}$ with
 \[
c_i := \sum_{j=1}^5(-1)^ja_{ij}b_{1\dots\check{\jmath}\dots5}.
 \]			
For instance,
$c_1= a_{12}b_{1345} - a_{13}b_{1245} + a_{14}b_{1235} - a_{15}b_{1234}$.
Then we look for relations between the generators $\bc$ of $S_2$,
defining the third syzygy module $S_3$. It is generated by one element
 \[
\begin{aligned}
d &:= (a_{23}a_{45} - a_{24}a_{35} + a_{25}a_{34})c_1 + (-a_{13}a_{45} + a_{14}a_{35} - a_{15}a_{34})c_2\\
&+(a_{12}a_{45} - a_{14}a_{25} + a_{15}a_{24})c_3 + (-a_{12}a_{35} + a_{13}a_{25} - a_{15}a_{23})c_4\\
&+(a_{12}a_{34} - a_{13}a_{24} + a_{14}a_{23})c_5 = 0.
\end{aligned}
 \]				
Thus the minimal free resolution of $\A_G^5$ is
(note that here, as well as in our other examples, the length of
the resolution is smaller than what the Hilbert theorem predicts):
 \[
\R[\bx,\by] \supset \A_G^5 \gets \R[\ba] \gets
\cR[\bb] \gets \cR[\bc] \gets \cR[\bd] \gets 0.
 \]

As before, to generate the field of rational invariants, we express
superfluous generators in terms of the others using the first syzygies. Namely, we express $a_{34},a_{35},a_{45}$ from the relations
$b_{1234},b_{1235},b_{1245}$; the other 2 syzygies follow from the
higher syzygies.
Removing these generators we obtain a set of 7 independent generators
$\bar{\ba} = \ba \setminus \{a_{34}, a_{35}, a_{45}\}$ whence
 \[
\R(\bx,\by)\supset\cF_G^5\simeq\R(\bar{\ba}) \overset{7}{\supset}\R.
 \]

\subsubsection{General $V^{\times m}$}

The previous arguments generalize straightforwardly to conclude
that $\A_G^m$ is generated by $\ba=\{a_{ij}:1\leq i<j\leq m\}$.
The first syzygy module is generated by the Pl\"ucker relations
$\bb=\{b_{ijkl}:1\leq i<j<k<l\leq m\}$. In other words we have:
 \[
\A_G^m=\langle\ba\,|\,\bb\rangle.
 \]

Similarly, the field of rational invariants is generated by $\ba$,
yet all of them except for $a_{1j},a_{2j}$ can be expressed (rationally)
through the rest via the Pl\"ucker relations $b_{12kl}$.
Denote $\bar{\ba}:=\{a_{12},a_{13},\dots,a_{1m},a_{23},\dots,a_{2m}\}$,
$\#\bar{\ba}=2m-3$. Then we get for $m\ge2$:
 \[
\R(\bx,\bp)\supset\cF_G^m\simeq\R(\bar{\ba})\overset{2m -3}{\supset}\R.	
 \]

\subsection{The general case: algebra of polynomial invariants}

Minimal free resolutions can be computed in many examples for $n\ge1$.
However, in what follows we restrict our attention to describing
generators/relations of $\A_G^m$.

Let us count the number of local smooth invariants. The action of
$G$ on $V$ is almost transitive, so the stabilizer of a nonzero point
$A_1$ has $\dim G_{A_1}=\binom{2n+1}2-2n=\binom{2n}2$.
For a generic $A_2$ there is only one invariant $a_{12}$
(the orbit has codimension 1) and
the stabilizer of $A_2$ in $G_{A_1}$ has
$\dim G_{A_1,A_2}=\binom{2n}2-(2n-1)=\binom{2n-1}2$.
For a generic $A_3$ there are two more new invariants $a_{13},a_{23}$
(the orbit has codimension $2+1=3$)
and the stabilizer of $A_3$ in $G_{A_1,A_2}$ has
$\dim G_{A_1,A_2,A_3}=\binom{2n-1}2-(2n-2)=\binom{2n-2}2$.
By the same reason for $k\leq 2n$ the stabilizer of a generic
$k$-tuple of points $A_1,\dots,A_k$ has
$\dim G_{A_1,\dots,A_k}=\binom{2n-k+1}2$. Finally for $k=2n$
the stabilizer of generic $A_1,\dots,A_{2n}$ is trivial.

Thus we get the expected number of invariants $a_{ij}$.
For $m\leq 2n+1$ there are no relations between them, and the first
comes at $m=2n+2$. These can be obtained by successively studying cases of increasing $n$ resulting in the {\em Pfaffian relation}:
 $$
b_{i_1i_2\dots i_{2n+1}i_{2n+2}}:=\op{Pf}(a_{i_pi_q})_{1\leq p,q\leq 2n+2} =0.
 $$

Recall that the Pfaffian of a skew-symmetric operator $S$
on $V$ with respect to $\omega$ is
$\op{Pf}(S)=\op{vol}_\omega(Se_1,\dots,Se_{2n})$
for any symplectic basis $e_i$ of $V$.
The properties of the Pfaffian are:
$\op{Pf}(S)^2=\det(S)$, $\op{Pf}(TST^t)=\det(T)\op{Pf}(S)$.
For $n=1$ we get
 \[
b_{1234}=
\op{Pf}\begin{pmatrix}
0 & a_{12} & a_{13} & a_{14} \\ -a_{12} & 0 & a_{23} & a_{24} \\
-a_{13} & -a_{23} & 0 & a_{34} \\ -a_{14} & -a_{24} & -a_{34} & 0
\end{pmatrix}=a_{12}a_{34}-a_{13}a_{24}+a_{14}a_{23}.
 \]
Similarly, for $n=2$ we get
 \begin{equation*}	
 \begin{aligned}
b_{123456}= &a_{12}a_{34}a_{56} - a_{12}a_{35}a_{46} + a_{12}a_{36}a_{45} - a_{13}a_{24}a_{56} + a_{13}a_{25}a_{46} -a_{13}a_{26}a_{45} +\\
&a_{14}a_{23}a_{56} - a_{14}a_{25}a_{36} + a_{14}a_{26}a_{35} - a_{15}a_{23}a_{46} + a_{15}a_{24}a_{36} - a_{15}a_{26}a_{34} + \\
&a_{16}a_{23}a_{45} - a_{16}a_{24}a_{35} + a_{16}a_{25}a_{34}=0.
 \end{aligned}
 \end{equation*}

Denote $\bb=\{b_{i_1i_2\dots i_{2n+1}i_{2n+2}}:
1\leq i_1<i_2<\dots<i_{2n+1}<i_{2n+2}\leq m\}$.

 \begin{theorem}\label{T1}
The algebra of $G$-invariants is generated by $\ba$ with
syzygies $\bb$:
 $$
\A_G^m=\langle\ba\,|\,\bb\rangle.
 $$
 \end{theorem}

 \begin{proof}
Let us first prove that the invariants $a_{ij}$ generate the field
$\cF_G^m$ of rational invariants for $m=2n$. We use the symplectic analog of
Gram-Schmidt normalization: 
given  points $A_1,\dots,A_{2n}$ in general position,
we normalize them using $G=\op{Sp}(2n,\R)$ as follows.

Let $e_1,\dots,e_{2n}$ be a symplectic basis of $V$, i.e.\ 
$\omega(e_{2k-1},e_{2k})=1$ and $\omega(e_i,e_j)=0$ else.
At first $A_1$ can be mapped to the vector $e_1$.
The point $A_2$ can be mapped to the line $\R e_2$, and 
because of $\omega(OA_1,OA_2)=a_{12}$
it is mapped to the vector $a_{12}e_2$. Next in mapping $A_3$
we have two constraints $\omega(OA_1,OA_3)=a_{13}$, $\omega(OA_2,OA_3)=a_{23}$, and the point can be mapped to the space
spanned by $e_1,e_2,e_3$ satisfying those constraints.
Continuing like this we arrive to the following matrix with columns $OA_i$:
 $$
\begin{pmatrix}
1 & 0 & -\frac{a_{23}}{a_{12}} & -\frac{a_{24}}{a_{12}} & \dots &
-\frac{a_{2,2n-1}}{a_{12}} & -\frac{a_{2,2n}}{a_{12}} \\
0 & a_{12} & a_{13} & a_{14} & \dots & a_{1,2n-1} & a_{1,2n} \\
0 & 0 & 1 & 0 & \dots & * & * \\
0 & 0 & 0 & \frac{b_{1234}}{a_{12}} & \vdots & * & * \\
\vdots & \vdots & \vdots & \vdots & \ddots & \vdots & \vdots \\
0 & 0 & 0 & 0 & \dots & 1 & 0 \\
0 & 0 & 0 & 0 & \dots & 0 & a_{2n-1,2n}
\end{pmatrix}
 $$
where $b_{1234}=a_{12}a_{34}-a_{13}a_{24}+a_{14}a_{23}$
(this does not vanish in general if $n>1$) and
by $*$ we denote some rational expressions in $a_{ij}$
that do not fit the table.

If $m<2n$ then only the first $m$ columns of this matrix have to be kept.
If $m>2n$ then the remaining points $A_{2n+1},\dots,A_m$ have all their
coordinates invariant as the stabilizer of the first $2n$ points is
trivial. Thus the invariants are expressed rationally in $a_{ij}$.

To obtain polynomial invariants one clears the denominators
in these rational expressions, and so
$\A_G^m$ is generated by $\ba$ as well.

Now the Pfaffian of the skew-symmetric matrix $(a_{ij})_{2k\times 2k}$
is the square root of the determinant of the Gram matrix
of the vectors $OA_i$, $1\leq i\leq k$, with respect to $\omega$.
If we take $k=n+1$ then the vectors
are linearly dependent and therefore the Pfaffian vanishes.
Thus $\bb$ are syzygies among the generators $\ba$.
That they form a complete set follows from the
same normalization procedure as above.
 \end{proof}

 \begin{remark}
Theorem \ref{T1} is basically known: 
H.\ Weyl described the generators $\ba$ as the first fundamental theorem;
his second fundamental theorem gives not only the syzygy denoted above
by $\bb$, but also several different Pfaffians of larger sizes. Namely he lists in \cite[VI.1]{W} the syzygies $b_{i_1\dots i_{2n+2k}}:=
\op{Pf}(a_{i_pi_q})_{1\leq p,q\leq 2n+2k}=0$, $1\leq k\leq n$.
Those however are abundant. For instance,
in the simplest case $n=2$
 $$
b_{12345678}=a_{12}b_{345678}-a_{13}b_{245678}+a_{14}b_{235678}
-a_{15}b_{234678}+a_{16}b_{234578}-a_{17}b_{234568}+a_{18}b_{234567}.
 $$
In general the larger Pfaffians can be expressed via the smallest
through the expansion by minors \cite{IW} (this fact was
also noticed in \cite{V}). Here is the corresponding Pfaffian identity
(below we denote $S_{2n+1}=\{\sigma\in S_{2n+2}:
\sigma(1)=1\}$)
 \[
b_{i_1i_2\dots i_{2n+1}i_{2n+2}}=
\frac1{n!}\sum_{\sigma\in S_{2n+1}}(-1)^{\op{sgn}(\sigma)}
a_{i_1i_{\sigma(2)}}b_{i_{\sigma(3)}\dots i_{\sigma(2n+2)}}.
 \]

In \cite[\S9.5]{PV} another set of syzygies was added:
$q_{i_1\dots i_{4n+2}}=\det(a_{i_s,i_{t+2n+1}})_{s,t=1}^{2n+1}=0$.
These are also abundant, and should be excluded.
For instance, for $n=1$ we get
$q_{123456}=
a_{12}b_{3456}-a_{34}b_{1256}+a_{35}b_{1246}-a_{36}b_{1245}$.
 \end{remark}

\subsection{The general case: field of rational invariants}\label{Finv}

Since $G$ is simple, the field of rational invariants is the field
of fractions of the algebra of polynomial invariants: $\cF_G^m=F(\A_G^m)$.
To obtain its basis one can use the syzygies $b_{i_1\dots i_{2n+2}}=0$
to express all invariants through 
$\bar{\ba}=\{a_{ij}:1\leq i\leq 2n;\,i<j\leq m\}$.

This can be done rationally (with $b_{1\dots 2n}\not\equiv0$
in the denominator), for instance for $n=2$
we can express $a_{56}$ from the syzygy $b_{123456}=0$ as follows:
 \begin{multline*}
a_{56} = (a_{12}a_{35}a_{46} - a_{12}a_{36}a_{45} - a_{13}a_{25}a_{46} + a_{13}a_{26}a_{45} + a_{14}a_{25}a_{36} - a_{14}a_{26}a_{35}
+ a_{15}a_{23}a_{46} \hphantom{aaa} \\
\hphantom{aaaaa} -  a_{15}a_{24}a_{36} + a_{15}a_{26}a_{34} - a_{16}a_{23}a_{45} + a_{16}a_{24}a_{35} - a_{16}a_{25}a_{34})
 /(a_{12}a_{34} -a_{13}a_{24} + a_{14}a_{23}).
 \end{multline*}

In general we have $\#\bar{\ba}= 2nm-n(2n+1)$ for for $m\ge 2n$, in summary:
\begin{framed}[200pt]
\vspace{-5pt}
 \[
\R(\bx,\by)\supset\cF_G^m\simeq\R(\bar{\ba}) \overset{d(m,n)}{\supset}\R,
 \]
\end{framed}
where
 $$
d(m,n)=\left\{
\begin{array}{ll} 2nm-n(2n+1) & \text{for }m\ge2n\\
\binom{m}2 &\text{for }m\leq2n.\end{array}\right.
 $$

\section{Variation on the group and space}

Let us consider inclusion of symmetrization, scaling and translations to
the transformation group $G$. We also discuss contactization of the action.

\subsection{Symmetric joint invariants}

Invariants of the extended group $\hat G=\op{Sp}(2n,\R)\times S_m$ on $V^{\times m}$
are equivalent to $G$-invariants on configurations of unordered sets of points
$V^{\times m}/S_m$ (which is an orbifold).
Denote the algebra of polynomial $\hat G$-invariants on $V^{\times m}$
by $\cS^m_G\subset\A_G^m$. The projection $\pi:\A_G^m\to\cS^m_G$
is given by
 $$
\pi(f)=\frac1{m!}\sum_{\sigma\in S_m}\sigma\cdot f.
 $$
As a Noetherian algebra $\cS^m_G$ is finitely generated, yet
it is not easy to establish its generating set explicitly.
All linear terms average to zero, $\pi(a_{ij})=0$, but there are several
invariant quadratic terms in terms of the
homogeneous decomposition $\A_G^m=\oplus_{k=0}^\infty\A^m_k$.

For example, for $n=1$, $m=4$ we have $\A^4_0=\R$,
$\A^4_1=\R^6=\langle a_{12},a_{13},a_{14},a_{23},a_{24},a_{34}\rangle$,
$\A^4_2=\R^{20}$ (21 monomials $a_{ij}a_{kl}$
modulo 1 Pl\"ucker relation), etc.
Then $\pi(\A^4_0)=\R$, $\pi(\A^4_1)=0$, and $\pi(\A^4_2)=\R^2$ has
generators
 \begin{align*}
6\pi(a_{12}^2) \,&= a_{12}^2+a_{13}^2+a_{14}^2+a_{23}^2+a_{24}^2+a_{34}^2,\\
12\pi(a_{12}a_{13}) \,&= a_{12}a_{13}+a_{12}a_{14}+a_{13}a_{14}
-a_{12}a_{23}-a_{12}a_{24}+a_{23}a_{24}\\
&+a_{13}a_{23}-a_{13}a_{34}-a_{23}a_{34}
+a_{14}a_{24}+a_{14}a_{34}+a_{24}a_{34}.
 \end{align*}

 \begin{theorem}
The field of symmetric rational invariants $\mathfrak{F}^m_G=\pi(\cF^m_G)$
is the field of fractions $\mathfrak{F}^m_G=F(\cS^m_G)$ and its
transcendence degree is $d(m,n)$.
 \end{theorem}

 \begin{proof}
This follows from general theorems \cite[\S2.5]{Sp} and discussion in Section \ref{S2}.
 \end{proof}

The last statement can be made more constructive: Let $\ell$ numerate indices
$(ij)$ of the basis $\bar\ba$ of $\cF^m_G$ as in \S\ref{Finv}, $1\leq\ell\leq d=d(m,n)$.
One can check that $q_k=\pi(\prod_{\ell\leq k}a_\ell^2)$ 
are algebraically independent. Thus, denoting $\bq=(q_1,\dots,q_d)$ we obtain the presentation
 \begin{framed}[200pt]
\vspace{-5pt}
 \[
\R(\bx,\by)\supset\mathfrak{F}_G^m\stackrel{\text{alg}}\supset \R(\bq) \overset{d(m,n)}{\supset}\R.
 \]
 \end{framed}

Here is an algorithm to obtain generators of $\cS^m_G$.

 \begin{proposition}
Fix an order on generators $a_{ij}$ of $\A^m_G$, and induce the
total lexicographic order on monomials $a^\sigma\in\mathcal{R}=\R[\ba]$.
Let $\Sigma$ be the Gr\"obner basis of the 
$\cR$-ideal generated by $\pi(a^\sigma)$.
Then elements $\pi(a^\sigma)$, contributing to $\Sigma$, generate $\cS^m_G=\pi(\A^m_G)$.
 \end{proposition}

 \begin{proof}
Note that the algorithm proceeds in total degree of $a^\sigma$ until
the Gr\"obner basis stabilizes. 
That the involved $\pi(a^\sigma)$ generate $\cS^m_G$ as an algebra
(initially they generate the ideal $\cR\cdot\pi(\A^m_G)\subset\A^m_G$)
follows from the same argument as in the proof of Hilbert's theorem on invariants \cite{H}. (The above $\pi$ is the Reynolds operator used there.)
 \end{proof}

Let us illustrate how this works in the first nontrivial case $m=3$,
for any $n$.

In this case the graded components of $\cS_G^3=\pi(\A_G^3)$
have the following dimensions:
$\dim\cS^3_0=1$, $\dim\cS^3_1=0$, $\dim\cS^3_2=2$, $\dim\cS^3_3=1$,
$\dim\cS^3_4=4$, $\dim\cS^3_5=2$, $\dim\cS^3_6=7$, etc,
encoded into the Poincar\'e series
 $$
P^3_{\cS}(z)=1+2z^2+z^3+4z^4+2z^5+7z^6+4z^7+10z^8+7z^9+{}\dots{}
 =\frac{1+z^4}{(1-z^2)^2(1-z^3)}.
 $$

For the monomial order $a_{12}>a_{13}>a_{23}$ the invariants
 \begin{gather*}
I_{2a} = 3\pi(a_{12}^2)= a_{12}^2+a_{13}^2+a_{23}^2,\quad
I_{2b} = 3\pi(a_{12}a_{13})= a_{12}a_{13}-a_{12}a_{23}+a_{13}a_{23},\\
I_3 = 6\pi(a_{12}^2a_{13})= a_{12}^2(a_{13}+a_{23})-a_{23}^2(a_{12}+a_{13})+a_{13}^2(a_{12}-a_{23}),\\
I_4 = 3\pi(a_{12}^2a_{13}^2)= a_{12}^2a_{13}^2+
a_{12}^2a_{23}^2+a_{13}^2a_{23}^2
 \end{gather*}
generate a Gr\"obner basis of the ideal $\cR\cdot\pi(A^m_G)$
with the leading monomials of the corresponding Gr\"obner basis equal:
$a_{12}^2$, $a_{12}a_{13}$, $a_{13}^3$, $a_{12}a_{23}^3$,
$a_{13}^2a_{23}^2$, $a_{13}a_{23}^3$, $a_{23}^4$.

The Gr\"obner basis also gives the following syzygy $R_8$:
 $$
(4I_{2a}^2+4I_{2a}I_{2b}+3I_{2b}^2)I_{2b}^2
-(8I_{2a}^2+4I_{2a}I_{2b}+14I_{2b}^2)I_4+4(I_{2a}-2I_{2b})I_3^2+27I_4^2=0.
 $$
In other words, $\cS^m_G=\langle I_{2a},I_{2b},I_3,I_4\,|\,R_8\rangle$.
We also derive a presentation of the field of rational invariants
($2:1$ means quadratic extension)
 $$
\R(\bx,\by) \supset\mathfrak{F}_G^3 
\overset{2:1}{\supset}\R(I_{2a},I_{2b},I_3)\overset{3}{\supset}\R.	
 $$

\subsection{Conformal and affine symplectic groups}

For the group $G_1=\op{CSp}(2n,\R)=\op{Sp}(2n,\R)\times\R_+$
the scaling makes the
invariants $a_{ij}$ relative, yet of the same weight, so their ratios
$[a_{12}:a_{13}:\dots:a_{m-1,m}]$ or simply the invariants
$I_{ij}=\frac{a_{ij}}{a_{12}}$ are absolute invariants.
These generate the field of invariants of transcendence degree
$d(m,n)-1$.

\smallskip

For the group $G_2=\op{ASp}(2n,\R)=\op{Sp}(2n,\R)\ltimes\R^{2n}$
the translations do not preserve the origin $O$ and this makes
$a_{ij}$ non-invariant. However due to the formula
2$\omega(A_1A_2A_3)=a_{12}+a_{23}-a_{13}$
(or more symmetrically: $a_{12}+a_{23}+a_{31}$),
with the proper orientation of the triangle $A_1A_2A_3$,
we easily recover the absolute invariants
$a_{ij}+a_{jk}+a_{ki}$.

Alternatively, using the translational freedom, we can move the point
$A_1$ to the origin $O$. Then its stabilizer in $G_2$ is $G=\op{Sp}(2n,\R)$
and we compute the invariants of $(m-1)$ tuples of points
$A_2,\dots,A_m$ as before. In particular they
generate the field of invariants of transcendence degree
$d(m-1,n)$.

\subsection{Invariants in the contact space}\label{cont}

Infinitesimal symmetries of the contact structure $\Pi=\op{Ker}(\alpha)$, $\alpha=du-\by\,d\bx$ in the contact space $M=\R^{2n+1}(\bx,\by,u)$, where
$\bx=(x_1,\dots,x_n)$, $\by=(y_1,\dots,y_n)$, are given by the contact vector field $X_H$ with the generating function $H=H(\bx,\by,u)$.
Taking quadratic functions $H$ with weights $w(\bx)=1$, $w(\by)=1$, $w(u)=2$ results in the conformally symplectic Lie algebra, which
integrates to the conformally symplectic group $G_1=\op{CSp}(2n,\R)$ (taking $H$ of degree $\leq2$ results in the affine extension of it by the
Heisenberg group).

Alternatively, one considers the natural lift of the linear action of $G=\op{Sp}(2n,\R)$ on $V=\R^{2n}$ to the contactization $M$
and makes a central extension of it. We will discuss the invariants of this action. Note that this action is no longer linear, so the
invariants cannot be taken to be polynomial, but can be assumed rational.

\subsubsection{The case $n=1$}
In the 3-dimensional case the group $G_1=\op{GL}(2,\R)$ acts on $M=\R^3(x,y,u)$ as follows:
 \begin{multline*}
G_1\ni g=\begin{pmatrix}\alpha & \beta\\ \gamma & \delta \end{pmatrix}:
(x,y,u)\mapsto (\alpha x + \beta y, \gamma x + \delta y, f(x,y,u)),\\
\text{where } f(x,y,u)=(\alpha \delta - \beta \gamma)\left(u - \frac{xy}{2}\right) + \frac{(\alpha x + \beta y)(\gamma x + \delta y)}{2}.\quad
 \end{multline*}

This action is almost transitive (no invariants), however
there are singular orbits and a relative invariant $R=xy-2u$.
Extending the action to multiple copies of $M$, i.e.\
considering the diagonal action of $G_1$ on $M^{\times m}$,
results in $m$ copies of this relative invariant, but also in the lifted invariants from various $V^{\times 2}$:
 $$
R_k=x_ky_k-2u_k\ (1\leq k\leq m),\quad R_{ij}=x_iy_j-x_jy_i\ (1\leq i<j\leq m).
 $$
These are all relative invariants of the same weight, therefore their ratios are absolute invariants:
 $$
T_k=\frac{R_k}{R_m}\ (1\leq k< m),\quad
T_{ij}=\frac{R_{ij}}{R_m}\ (1\leq i<j\leq m).
 $$
Since $u_k$ enter only $R_k$ there are no relations involving those, and the relations on $T_{ij}$ are the same as for $a_{ij}$,
namely they are Pl\"ucker relations (since those are homogeneous, they are satisfied by both $R_{ij}$ and $T_{ij}$).
As previously, we can use them to eliminate all invariants except for $\bar{\bT}=\{T_k,T_{1i},T_{2i}\}$:
 \[
T_{kl}=\frac{T_{1k}T_{2l}-T_{1l}T_{2k}}{T_{12}},\quad 3\leq k<l\leq m.
 \]
The field of rational invariants for $m>1$ is then described as follows:
 \[
\R(x,y,u) \supset \cF_{G_1}^m \simeq \R(\bar{\bT}) \overset{3m - 4}{\supset} \R.	
 \]

\subsubsection{The general case}

In general we also have no invariants on $M$ and the following relative invariants on $M^{\times m}$
 $$
R_k=\bx_k\by_k-2u_k\ (1\leq k\leq m),\quad R_{ij}=\bx_i\by_j-\bx_j\by_i\ (1\leq i<j\leq m)
 $$
resulting in absolute invariants $T_k,T_{ij}$ given by the same formulae.
Again using the Pfaffian relations we can rationally eliminate
superfluous generators, and denote the resulting set by $\bar{\bT}=\{T_k,T_{ij}:1\leq k<m,i<j\leq m,1\leq i\leq 2n\}$.
This set is independent and contains $\bar{d}(m,n)$ elements, where
 $$
\bar{d}(m,n)=\left\{
\begin{array}{ll} (2n+1)m-n(2n+1)-1 & \text{for }m\ge2n\\
\binom{m}2+m-1=\binom{m+1}2-1 &\text{for }m\leq2n.\end{array}\right.
 $$

This $\bar{d}(m,n)$ is thus the transcendence degree of the field of rational invariants:
\begin{framed}[200pt]
\vspace{-5pt}
 \[
\R(\bx,\by,u) \supset \cF_{G_1}^m \simeq \R(\bar{\bT}) \overset{\bar{d}(m,n)}{\supset} \R.	
 \]
\end{framed}

\section{From joint to differential invariants}

When we pass from finite to continuous objects the equivalence
problem is solved through differential invariants.
In \cite{JK} this was done for submanifolds and functions with respect
to our groups $G$. After briefly recalling the results, we will demonstrate how to perform the discretization in several different cases.

\subsection{Jets of curves in symplectic vector spaces}

Locally a curve in $\R^{2n}$ is given
as $\bu=\bu(t)$ for $t=x_1$ and $\bu=(x_2,\dots,x_n,y_1,\dots,y_n)$
in the canonical coordinates $(x_1,x_2,\dots,x_n,y_1,\dots,y_n)$,
$\omega=dx_1\wedge dy_1+\dots+dx_n\wedge dy_n$.
The corresponding jet-space $J^\infty(V,1)$ has coordinates
$t,\bu,\bu_t,\bu_{tt},\dots$, and $J^k$ is the truncation of it.
For instance, $J^1(V,1)=\R^{4n-1}(t,\bu,\bu_t)$.
Note that $\dim J^k(V,1)=2n+k(2n-1)$.

\smallskip

In the case of dimension $2n=2$, the jet-space is
$J^k(V,1)=\R^{k+2}(x,y,y_x,\dots,y_{x..x})$.
Here $G=\op{Sp}(2,\R)$ has an open orbit in $J^1(V,1)$,
and the first differential invariant is of order 2:
 $$
I_2=\frac{y_{xx}}{(xy_x-y)^3}.
 $$
There is also an invariant derivation ($\mathcal{D}_x$ is the total derivative wrt $x$)
 $$
\nabla=\frac1{xy_x-y}\mathcal{D}_x.
 $$
By differentiation we get new differential invariants
$I_3=\nabla I_2$, $I_4=\nabla^2 I_2$, etc. The entire algebra
of differential invariants is free:
 $$
\mathcal{A}_G=\la I_2\,;\,\nabla\ra.
 $$

In the general case we denote the canonical coordinates
on $V=\R^{2n}$ by $(t,\bx,y,\bz)$, where $\bx$ and $\bz$ and
$(n-1)$-dimensional vectors. $G=\op{Sp}(2n,\R)$ acts on $J^\infty(V,1)$. 
The invariant derivation is equal to
 $$
\nabla=\frac1{(ty_t-y+\bx\bz_t-\bx_t\bz)}\mathcal{D}_t.
 $$
and the first differential invariant of order 2 is
 $$
I_2=\frac{\bx_t\bz_{tt}-\bx_{tt}\bz_t+y_{tt}}{(ty_t-y+\bx\bz_t-\bx_t\bz)^3}.
 $$
There is one invariant $I_3$ of order 3 independent of $I_2,\nabla(I_2)$,
one invariant $I_4$ of order 4 independent of $I_2,\nabla(I_2),I_3,
\nabla^2(I_2),\nabla(I_3)$, and so on up to order $2n$.
Then the algebra of differential invariants of $G$ is freely generated \cite{JK} so:
 $$
\mathcal{A}_G=\la I_2,I_3,\dots,I_{2n}\,;\, \nabla\ra.
 $$

\subsection{Symplectic discretization}\label{sympldscr}

Consider first the case $n=1$ with coordinates $(x,y)$ on
$V=\R^2$. Let $A_i=(x_i,y_i)$, $i=0,1,2$, be three close points lying
on the curve $y=y(x)$. We assume $A_1$ is in between $A_0,A_2$ and
omit indices for its coordinates, i.e.\ $A_1=(x,y)$.

Let $x_0=x-\delta$ and $x_2=x+\epsilon$.
Denote also $y'=y'(x)$, $y''=y''(x)$, etc.
Then from the Taylor formula we have:
 \begin{gather*}
y_0=y-\delta y'+\tfrac12\delta^2y''-\tfrac16\delta^3y'''+o(\delta^3), \\
y_2=y+\epsilon y'+\tfrac12\epsilon^2y''+\tfrac16\epsilon^3y'''+o(\epsilon^3).
 \end{gather*}
Therefore the symplectic invariants $a_{ij}=x_iy_j-x_jy_i$ are:
 \begin{align*}
a_{12} \,&= \epsilon(xy'-y)+\tfrac12\epsilon^2 xy''+
\tfrac16\epsilon^3xy'''+o(\epsilon^3), \\
a_{01} \,&= \delta(xy'-y)-\tfrac12\delta^2 xy''+
\tfrac16\delta^3xy'''+o(\delta^3), \\
a_{02} \,&= (\epsilon+\delta)(xy'-y)+\tfrac12(\epsilon^2-\delta^2)xy''\\
& +\tfrac16(\epsilon^3+\delta^3)xy'''
-\tfrac12(\epsilon+\delta)\epsilon\delta y''+o((|\delta|+|\epsilon|)^3).
 \end{align*}
This implies:
 $$
\frac{a_{01}-a_{02}+a_{12}}{a_{01}a_{02}a_{12}}=
\frac12\frac{y''}{(xy'-y)^3}+o(|\delta|+|\epsilon|).
 $$
Thus we can extract the invariant exploiting no distance
(like $\epsilon=\delta$) but only the topology ($\epsilon,\delta\to0$)
and the symplectic area. This works in any dimension, and using the coordinates
from the previous subsection we get
 $$
\lim_{A_0,A_2\to A_1}\frac{\op{Area}_\omega(A_0A_1A_2)}
{\op{Area}_\omega(OA_0A_1)\op{Area}_\omega(OA_0A_2)\op{Area}_\omega(OA_1A_2)}
=\frac{2(\bx_t\bz_{tt}-\bx_{tt}\bz_t+y_{tt})}
{(ty_t-y+\bx\bz_t-\bx_t\bz)^3}=2I_2.
 $$

Similarly we obtain the invariant derivation (it uses only two points
and hence is of the first order)
 $$
\lim_{A_0\to A_1}\frac{\overrightarrow{A_0A_1}}{\op{Area}_\omega(OA_0A_1)}
=\frac{2\mathcal{D}_t}{(ty_t-y+\bx\bz_t-\bx_t\bz)}=2\nabla.
 $$
The other generators $I_3,I_4,\dots$ (important for $n>1$)
can be obtained by a higher order discretization,
but the formulae become more involved.

\subsection{Contact discretization}

Now we use joint invariants to obtain differential invariants of curves
in contact 3-space $W=\R^3(x,y,u)$ with respect to the group
$G=\op{GL}(2,\R)$, acting as in \S\ref{cont}.
The curves will be given as $y=y(x),u=u(x)$ and their jet-space is $J^k(W,1)=\R^{2k+3}(x,y,u,y_x,u_x,\dots,y_{x..x},u_{x..x})$.
The differential invariants are generated in the Lie-Tresse sense
\cite{JK} as
 $$
\mathcal{A}_G=\la I_1,I_2\,;\,\nabla\ra.
 $$
where
 $$
I_1=\frac{u_x-y}{xy_x-y}\,,\quad I_2=\frac{(xy-2u)^2}{(xy_x-y)^3}\,y_{xx}\,,
\quad\nabla=\frac{xy-2u}{xy_x-y}\,\mathcal{D}_x.
 $$

Instead of exploiting the absolute rational invariants $T_i,T_{ij}$
we will work with the relative polynomial invariants $R_i,R_{ij}$ from
\S\ref{cont}. To get absolute invariants we will then have to pass to
weight zero combinations.

Consider three close points $\hat A_i=(x_i,y_i,u_i)$, $i=0,1,2$,
lying on the curve. We again omit indices for the middle point, so
$x_0=x-\delta$, $x_1=x$ and $x_2=x+\epsilon$. Using the Taylor decomposition
as in the preceding subsection, we obtain
 \begin{alignat*}{2}
& R_1= xy-2u,\quad &&
R_0-R_1=\delta(2u'-y-xy')+o(\delta),\\
& R_{01}=\delta(xy'-y)+o(\delta),\quad &&
R_{02}=(\epsilon+\delta)(xy'-y)+o(|\epsilon|+|\delta|),\\
& R_{12}=\epsilon(xy'-y)+o(\epsilon),\quad &&
R_{01}+R_{12}-R_{02}=\tfrac12\epsilon\delta(\epsilon+\delta)y''+
o((|\epsilon|+|\delta|)^3)
 \end{alignat*}
as well as
 $$
\overrightarrow{A_0A_1}=\delta(\p_x+y'\p_y+u'\p_y)+o(\delta).
 $$
Passing to jet-notations,
we obtain the limit formulae for basic differential invariants:
 \begin{gather*}
I_1=\lim_{A_0\to A_1}\frac{R_0-R_1}{2R_{01}}+\frac12=
\lim_{A_0\to A_1}\frac{T_0-1+T_{01}}{2T_{01}},\\
\frac12I_2=\lim_{A_0,A_2\to A_1}
\frac{R_1^2(R_{01}+R_{12}-R_{12})}{R_{01}R_{02}R_{12}}
=\lim_{A_0,A_2\to A_1}\frac{T_{01}+T_{12}-T_{12}}{T_{01}T_{02}T_{12}},\\
\nabla=\lim_{A_0\to A_1}\frac{R_1}{R_{01}}\overrightarrow{A_0A_1}
=\lim_{A_0\to A_1}\frac{\overrightarrow{A_0A_1}}{T_{01}}.
 \end{gather*}

These formulae straightforwardly generalize to invariants of jets of curves
in contact manifolds of dimension $2n+1$, $n>1$, in which case there
are also other generators obtained by higher order discretizations.

\subsection{Functions and other examples}

Let us discuss invariants of jets of functions on the symplectic plane.
The action of $G=\op{Sp}(2,\R)$ on $J^0V=V\times\R(u)\simeq\R^3(x,y,u)$,
with $I_0=u$ invariant, prolongs to $J^\infty(V)=\R^\infty(x,y,u,u_x,u_y,u_{xx},u_{xy},u_{yy},\dots)$.
Note that functions can be identified as surfaces in $J^0V$ through
their graphs.

For any finite set of points $\hat A_k=(x_k,y_k,u_k)$ the values $u_k$
are invariant, and the other invariants $a_{ij}$ are obtained from the
projections $A_k=(x_k,y_k)$. In this way we get the basic first order
invariant (as before we omit indices $x_1=x$, $y_1=y$, $u_1=y$ for
the reference point $A_1$ in the right hand side)
 $$
I_1=\lim_{A_0,A_2\to A_1}
\frac{a_{01}(u_1-u_2)+a_{12}(u_1-u_0)}{a_{01}-a_{02}+a_{12}}=
xu_x+yu_y
 $$
as well as two invariant derivations
 \begin{equation*}
\nabla_1=\overrightarrow{OA_1}=x\mathcal{D}_x+y\mathcal{D}_y,\quad
\nabla_2=\lim_{A_0\to A_1}\frac{I_1}{a_{01}}\,\overrightarrow{A_0A_1}-
\frac{u_1-u_0}{a_{01}}\,\overrightarrow{OA_1}=u_x\mathcal{D}_y-u_y\mathcal{D}_x.
 \end{equation*}

To obtain the second order invariant
$I_{2c}=u_x^2u_{yy}-2u_xu_yu_{xy}+u_y^2u_{xx}$ let $A_0$ belong to
the line through $A_1$ in the direction $\nabla_2$
(this constraint reduces the second order formula to depend on only
two points), i.e.\ $A_0=(x+\epsilon u_y,y-\epsilon u_x)$, $A_1=(x,y)$.
Then $u_0-u_1=\frac{\epsilon^2}2I_{2c}+o(\epsilon^2)$,
$a_{01}=\epsilon I_1$ and letting $\epsilon\to0$ we obtain
 $$
\lim_{\begin{subarray}{c}A_0\to A_1\\ A_0A_1\|\nabla_2\end{subarray}}
\frac{u_0-u_1}{a_{01}^2}=\frac{I_{2c}}{2I_1^2}.
 $$
In the same way we get 
$I_{2a}= x^2u_{xx}+2xyu_{xy}+y^2u_{yy}$ and 
$I_{2b}= xu_yu_{xx}-yu_xu_{yy}+(yu_y-xu_x)u_{xy}$.
These however are not required as the algebra of differential invariants
is generated as follows \cite{JK} (for some differential syzygies $\mathcal{R}_i$)
 $$
\mathcal{A}_G=\la I_0,I_{2c}\,;\,\nabla_1,\nabla_2\,|\,\mathcal{R}_1,
\mathcal{R}_2,\mathcal{R}_3\ra.
 $$

Similarly one can consider surfaces in the contact 3-space
(with the same coordinates $x,y,u$ but different lift of
$\op{Sp}(2,\R)$ extended to $\op{GL}(2,\R)$) and
higher-dimensional cases.
The idea of discretization of differential invariants applies
to other problems treated~in~\cite{JK}.

\section{Relation to binary and higher order forms}

According to the Sylvester theorem \cite{S} a general binary form
$p\in\C[x,y]$ of odd degree $2m-1$ with complex coefficients
can be written as
 $$
p(x,y)=\sum_{i=1}^m(\alpha_i x+\beta_i y)^{2m-1}.
 $$
This decomposition is determined up to permutation of linear factors and
independent multiplication of each of them by a $(2m-1)$-th root of unity.

In other words, we have the branched cover of order $k_m=(2m-1)^mm!$
 $$
\times^m(\C^2)\to S^{2m-1}\C^2
 $$
and the deck group of this cover is $S_m\ltimes\Z_{2m-1}^{\times m}$.

Note that in the real case, due to uniqueness of the odd root of unity,
the corresponding cover over an open subset of the base
 $$
\times^m(\R^2)\to S^{2m-1}\R^2
 $$
has the deck group $S_m$.

With this approach the invariants of real binary forms are precisely
the joint symmetric invariants studied in this paper, and for
complex forms one has to additionally quotient by $\Z_{2m-1}^{\times m}$,
which is equivalent to passing from $a_{ij}$ to $a_{ij}^{2m-1}$
and other invariant combinations
(example for $m=4$: $a_{12}^3a_{13}^2a_{14}^2a_{23}^2a_{24}^2a_{34}^3$)
and subsequently averaging by the map $\pi$.

Other approaches to classification of binary forms, most importantly
through differential invariants \cite{Ol,BL}, can be related to
this via symplectic discretization.

 \begin{remark}
Note that the standard "root cover" $\C^{2m}\to S^{2m-1}\C^2$:
 $$
(a_0,a_1,\dots,a_{2m-1})\mapsto(p_0,p_1,\dots,p_{2m-1}),\quad
\sum_{i=0}^{2m-1}p_ix^iy^{2m-i-1}=a_0\prod_{i=1}^{2m-1}(x-a_iy)
 $$
has order $(2m-1)!<k_m$. Polynomial $\op{SL}(2,\C)$-invariants
of binary forms with this approach correspond to
functions on the orbifold $\C^{2m}/S_{2m}$.
 \end{remark}

The above idea extends further to ternary and higher valence forms
(see \cite{BL2} for the differential invariants approach
and \cite{PO} for an approach using joint differential invariants)
with the Waring decompositions \cite{AH} as the cover,
but here the group $G$ is no longer symplectic.
We expect all the ideas of the present paper to generalize
to the linear and affine actions of other reductive groups $G$.

\bigskip

\textsc{Acknowledgements.}
FA thanks J\o rn Olav Jensen for stimulating conversations and feedback.
BK thanks Pavel Bibikov, Eivind Schneider and Boris Shapiro for
useful discussions.
The publication charges for this article have been funded by a grant
from the publication fund of UiT the Arctic University of Norway.



\begin{thebibliography}{WW}
 \footnotesize

\bibitem{AH}
J.\ Alexander, A.\ Hirschowitz, {\it Polynomial interpolation in several
variables}, Journal of Algebraic Geometry {\bf 4} (2), 201-222 (1995).

\bibitem{FA}
F.\ Andreassen, {\it Joint Invariants of of Symplectic and Contact Lie Algebra Actions},
Master Thesis in Mathematics, hdl.handle.net/10037/19003,
UiT the Arctic University of Norway, June 2020.

\bibitem{BL}
P.\ Bibikov, V.\ Lychagin, {\it $\op{GL}_2(\C)$-orbits of Binary
Rational Forms},  Lobachevskii Journal of Mathematics {\bf 32}, no.\ 1,
95-102 (2011).

\bibitem{BL2}
P.\ Bibikov, V.\ Lychagin, {\it Classification of linear actions
of algebraic groups on spaces of homogeneous forms},
Doklady Mathematics {\bf 85} (1), 109-112 (2012).

\bibitem{E}
D.\ Eisenbud, {\it The geometry of syzygies: a second course in algebraic geometry
and commutative algebra}, Graduate Texts in Mathematics {\bf 229}, Springer (2006).

\bibitem{H}
D.\ Hilbert, {\it Theory of algebraic invariants}, Cambridge University Press (1993).


\bibitem{IW}
M.\ Ishikawa, M.\ Wakayama, {\it Minor summation formula of Pfaffians,
Survey and a new identity}, Adv.\ Stud.\ Pure Math.\ {\bf 28}, 133-142 (2000).

\bibitem{JK}
J.\,O.\ Jensen, B.\ Kruglikov,
{\it Differential Invariants of Linear Symplectic Actions},
arXiv: 2010.08024 (2020).

\bibitem{KL}
B. Kruglikov, V. Lychagin, {\it Global Lie-Tresse theorem}, Selecta Math.\
{\bf 22}, 1357-1411 (2016).

\bibitem{MFK}
D.\ Mumford, J.\ Fogarty, F.\ Kirwan, {\it Geometric invariant theory}, Springer (1994).

\bibitem{Ol}
P.\,J.\ Olver, {\it Classical invariant theory},
London Mathematical Society Student Texts {\bf 44},
Cambridge University Press (1999).

\bibitem{O2}
P.\,J.\ Olver, {\it Joint invariant signatures}, Found.\ Comput.\ Math.\ {\bf 1}(1), 3-68 (2001).

\bibitem{PO}
G.\ G\"un Polat, P.\,J.\ Olver, {\it Joint differential invariants of binary and ternary forms},
Portugaliae Math.\ {\bf 76}, 169-204 (2019).

\bibitem{PV}
V.\,L.\ Popov, E.\,B.\ Vinberg, {\it Invariant theory}, In:
Algebraic geometry IV, Encyclopaedia of Mathematical Sciences {\bf 55},
Springer-Verlag, Berlin (1994).

\bibitem{R}
M.\ Rosenlicht, {\it Some basic theorems on algebraic groups}, Amer.\
J.\ Math.\ {\bf 78} (2), 401-443 (1956).

\bibitem{Sp}
T.\,A.\ Springer, {\it Invariant Theory\/}, Lecture Notes in Math.\ {\bf 585}, Springer-Verlag (1977).

\bibitem{S}
J.\,J.\ Sylvester, {\it An Essay on Canonical Forms}
(published: George Bell 1851);
{\it On a remarkable discovery in the theory of canonical forms and of hyperdeterminants} (published: Philosophical Magazine 1851);
in Mathematical Papers {\bf 1}, Chelsea, New York,
34:\,203-216, 41:\,265-283 (1973).

\bibitem{V}
Th.\ Vust, {\it Sur la th\'eorie des invariants des groupes classiques}, Ann.\ Inst.\ Fourier {\bf 26}, no.\ 1, 1-31 (1976).

\bibitem{W}
H.\ Weyl, {\it Classical groups}, Princeton University Press (1946).

\end{thebibliography}
\end{document}